\documentclass{article} 
\usepackage[utf8]{inputenc}
\usepackage[T1]{fontenc}
\usepackage{graphicx}
\usepackage{newunicodechar}
\usepackage{graphicx,bxtexlogo,mathtools,amssymb,afterpage,here,physics,url,graphicx,color,amsthm,ascmac,tcolorbox,amsfonts,tikz}

\newtheorem{thm}{Theorem}
\newtheorem{prop}{Proposition}
\newtheorem{lem}{Lemma}
\newtheorem{cor}{Corollary}
\theoremstyle{definition}
\newtheorem{defi}{Definition}
\newtheorem{ex}{Example}
\theoremstyle{remark}
\newtheorem{rem}{Remark}
\newcommand{\bbr}{\mathbb{R}}
\newcommand{\bbe}{\mathbb{E}}
\newcommand{\cals}{\mathcal{S}}

\newcommand{\del}{\partial}

\usepackage{graphicx}
\usepackage{multirow}\usepackage{amsmath,amssymb,amsfonts}\usepackage{amsthm}\usepackage{mathrsfs}\usepackage[title]{appendix}\usepackage{xcolor}\usepackage{textcomp}\usepackage{manyfoot}\usepackage{booktabs}\usepackage{algorithm}\usepackage{algorithmicx}\usepackage{algpseudocode}\usepackage{listings}
\usepackage{graphicx}
\usepackage{epstopdf}
\usepackage{float}
\usepackage{tikz}
\usepackage{pgfplots}
\usepackage{pgfplotstable}
\usetikzlibrary{arrows,positioning,plotmarks,external,patterns,angles,decorations.pathmorphing,backgrounds,fit,shapes,graphs,calc,spy}
\pgfplotsset{compat=1.18}
\usepackage{graphicx}%
\usepackage{multirow}%
\usepackage{amsmath,amssymb,amsfonts}%
\usepackage{amsthm}%
\usepackage{mathrsfs}%
\usepackage[title]{appendix}%
\usepackage{xcolor}%
\usepackage{textcomp}%
\usepackage{manyfoot}%
\usepackage{booktabs}%
\usepackage{algorithm}%
\usepackage{algorithmicx}%
\usepackage{algpseudocode}%
\usepackage{listings}%
\usepackage{mathtools}
\usepackage{amssymb}
\usepackage{afterpage}
\usepackage{physics}
\usepackage{color}
\usepackage{tcolorbox}

\makeatletter
\let\@fnsymbol\@arabic
\makeatother

\begin{document}
\title{Flatness of location-scale-shape models under the Wasserstein metric}
\author{Ayumu Fukushi\footnote{Department of Mathematical Informatics, The University of Tokyo \& Statistical Mathematics Unit, RIKEN Center for Brain Science, e-mail: \texttt{ayumu.fukushi@riken.jp}}, \ 
Yoshinori Nakanishi-Ohno\footnote{Faculty of Culture and Information Science, Doshisha University, Kyotanabe, e-mail: \texttt{ynakanis@mail.doshisha.ac.jp}}, \ 
	Takeru Matsuda\footnote{Department of Mathematical Informatics, The University of Tokyo \& Statistical Mathematics Unit, RIKEN Center for Brain Science, e-mail: \texttt{matsuda@mist.i.u-tokyo.ac.jp}}}

\date{}

\maketitle

	\begin{abstract} 
    In Wasserstein geometry, one-dimensional location–scale models are flat both intrinsically and extrinsically-that is, they are curvature-free as well as totally geodesic in the space of probability distributions. In this study, we introduce a class of one-dimensional statistical models, termed the location–scale–shape model, which generalizes several distributions used in extreme-value theory. This model has a shape parameter that specifies the tail heaviness. We investigate the Wasserstein geometry of the location-scale-shape model and show that it is intrinsically flat but extrinsically curved.

    \textbf{Keywords:}
Wasserstein geometry, Wasserstein information matrix, location-scale-shape model, displacement interpolation
	\end{abstract}

\section{Introduction}\label{introduction}

In information geometry, a parametric family of probability distributions is called a statistical model. By viewing the parameters as coordinates, one regards the model as a manifold, equips it with various metrics and connections, and studies the resulting geometric properties; this viewpoint has many applications in statistics, machine learning, and related areas. In particular, the geometry induced by the Kullback–Leibler (KL) divergence has been extensively studied, and its close relationship to maximum likelihood estimation is well known  \cite{Amari and Nagaoka}.

The Wasserstein distance is a central notion of distance between probability distributions in statistics and machine learning \cite{Chewi,Peyre,Santambrogio}. It arises from the field of optimal transport, and the associated geometry is known as the Wasserstein geometry \cite{villani2,Villani}. 
The associated metric is also known as Otto metric \cite{Otto}.
In contrast to the KL divergence, the Wasserstein distance is a genuine metric (symmetric and satisfying the triangle inequality) and remains meaningful even when the supports of the distributions differ. Also, the Wasserstein distance reflects the metric structure of the underlying space through transportation cost.
Motivated by the fact that the Fisher information matrix arises as the second-order approximation to the KL divergence, Li and Zhao \cite{Li and Zhao} proposed the Wasserstein information matrix. The intrinsic geometry induced by the Wasserstein information matrix and its statistical implications have been studied recently \cite{Amari and Matsuda2,Ay,Chen and Li,Li and Montufar,Li and Rubio, Li and Zhao}.

Whereas the Wasserstein distance does not admit a closed form in general, it is obtained in closed form in location-scale models \cite{Gelbrich}. 
This is the basis of recent studies on the Wasserstein geometry of location-scale models and  its statistical application \cite{Amari and Matsuda1, Amari and Matsuda2, Matsuda and Strawderman}. 
Geometrically, location-scale models are intrinsically flat as Riemannian manifolds under the metric proposed by \cite{Li and Zhao}.
Moreover, they are totally geodesic with respect to the $\mathrm{L}^2$-Wasserstein distance (i.e., closed under displacement interpolation). 

In this paper, we introduce the class of location-scale-shape models as an extension of the location-scale models and study its Wasserstein geometry. The shape parameter we introduce controls the tail heaviness, which is often the focus of statistical analysis. In particular, the generalized extreme value and generalized Pareto distributions—central to extreme value statistics, which concerns inference on rare events \cite{Haan and Ferreira}—are not location-scale models but are  the location-scale-shape models. 
We show that, as in earlier work on location–scale models, location–scale–shape models are intrinsically flat, however they are not extrinsically flat.

This paper is organized as follows. In Section 2, we give a brief review of the framework by Li and Zhao \cite{Li and Zhao} on the Wasserstein scores and the Wasserstein information matrix. 
In Section 3, we introduce the location-scale-shape model and derive its Wasserstein score functions and the Wasserstein information matrix. In Section 4, we discuss the Wasserstein–geometric properties of the location-scale-shape models. In particular, although the intrinsic curvature induced by the Wasserstein information matrix vanishes, the models is not totally geodesic with respect to the Wasserstein distance, i.e., it is not extrinsically flat when embedded in the Wasserstein space.

\section{Preliminaries}\label{short_review}
\subsection{Wasserstein information matrix}
On a measure space $(\mathfrak{X}, \mathfrak{B}, \nu)$, a family of probability density functions $\mathcal{S}$ with parameter space $\Theta$, an open subset of $\mathbb{R}^n$, given by
\begin{equation*}
  \mathcal{S}=\lbrace p(x;\theta)\mid \theta \in \Theta \rbrace,
\end{equation*}
is called a statistical model. In information geometry, a statistical model is regarded as a manifold with $\Theta$ as a system of local coordinates, and one studies the relationship between geometric properties and the statistical properties of models and estimators. 

Let $(\mathfrak{X},d)$ be a complete separable metric space and let $p\ge 1$. Denote by $\mathcal{P}_p(\mathfrak{X})$ the set of all Borel probability measures on $(\mathfrak{X},d)$ with finite $p$-th moment:
\begin{align}
  \mathcal{P}_p(\mathfrak{X})
  =\bigl\{\pi:\text{Borel probability measure on }(\mathfrak{X},d)\ \bigm|\ \int_{\mathfrak{X}} d(x,x_0)^p\,\dd\mu(x)<\infty\,\bigr\}.
\end{align}
We define the $L^p$-Wasserstein distance $W_p(\mu_1,\mu_2)$ on $\mathcal{P}_p(\mathfrak{X})$ by
\begin{align}\label{Wasserstein_distance}
W_p(\mu_1,\mu_2)^p:=\inf_{\pi\in\Pi(\mu_1,\mu_2)}\int_{\mathfrak{X}\times\mathfrak{X}} d(x,y)^p\dd\pi(x,y),
\end{align}
where
\begin{align*}
  \Pi(\mu_1,\mu_2)=\Bigl\{\, \pi \in \mathcal{P}_k(\mathfrak{X} \times \mathfrak{X}) \ \Bigm|\ 
  \pi(A \times \mathfrak{X}) = \mu_1(A),\ \pi(\mathfrak{X} \times B) = \mu_2(B)\ (A,B \in \mathcal{B}(\mathfrak{X})) \Bigr\}.
\end{align*}

In what follows, we focus on the space $\mathcal{P}_2(\mathbb{R}^d)$, taking $\mathfrak{X}=\bbr^d$, $p=2$, and $d$ to be the Euclidean distance on $\bbr^d$.
Li and Zhao \cite{Li and Zhao} introduced analogues of the Fisher score function
and Fisher information matrix 
on a statistical model $\mathcal{S}\subset\mathcal{P}_2(\mathbb{R}^d)$ under the $\mathrm{L}^2$-Wasserstein distance.
The Wasserstein score functions $\Phi^W_i(x;\theta)$, $i=1,\dots,n$, are defined as the solutions to the Poisson-type partial differential equation
\begin{align}
    \nabla_x \cdot \bigl(p(x;\theta)\,\nabla_x \Phi^W_i(x;\theta)\bigr)&=-\frac{\partial}{\partial \theta_i} p(x;\theta), \label{conti_eq}  \\
    \mathbb{E}_\theta[\Phi^W_i(X;\theta)]&=0, \label{normal_condi}
\end{align}
where $\nabla_x$ denotes the gradient with respect to $x$ and $\nabla_x \cdot$ denotes the divergence operator.
Note that equation~\eqref{conti_eq} is the parametric analogue of the continuity equation which, in optimal transport theory, characterizes geodesics in the $\mathcal{P}_2(\mathbb{R}^d)$.
In a related direction, \cite{Chen and Li} introduced Wasserstein score functions for the case where $\mathfrak{X}$ is a discrete space equipped with a undirected graph structure.
The Wasserstein score functions are analogous to the Fisher score functions $\frac{\del \log p(x;\theta)}{\del\theta^i}$. We denote the Fisher score functions by $\Phi^F_i(x;\theta)$.

We define the Wasserstein information matrix $I^W(\theta)$ by
\begin{align}\label{def_w_inf}
I^W(\theta)_{ij}
=\int \Phi^W_i(x;\theta)\,\frac{\del p(x;\theta)}{\del\theta^j}\, \dd x
=\bbe_\theta\!\left[\nabla_x\Phi^W_i(x;\theta)^\top \nabla_x\Phi^W_j(x;\theta)\right].
\end{align}
Same as the Fisher metric induced by the Fisher information matrix, the Wasserstein information matrix induces a Riemannian metric on $\cals$. We call this the Wasserstein metric and denote it by $g_W$. This metric is also known as the Otto metric \cite{Otto}.
The Wasserstein information matrix is analogous to the Fisher information matrix $I^F(\theta)_{ij}=\int \Phi^F_i(x;\theta)\frac{\del p(x;\theta)}{\del\theta^j}\dd x$.

The Wasserstein information matrix provides a second-order approximation of the $\mathrm{L}^2$-Wasserstein distance:
\begin{align*}
  W_2(p_\theta, p_{\theta+\Delta\theta})
  = \tfrac{1}{2}\,\Delta\theta^\top I^W(\theta)\,\Delta\theta + o(\|\Delta\theta\|^2).
\end{align*}
This parallels the classical fact that the Fisher information matrix is the second-order approximation of the Kullback–Leibler divergence. 
Li and Zhao \cite{Li and Zhao} also introduced an analogue of the maximum likelihood estimator and a new covariance notion; they derived a Wasserstein-–Cramér-–Rao inequality showing that the inverse of $I^W(\theta)$ gives a lower bound on the estimation variance. We omit the details here.

\subsection{Location-scale model}
Let $f$ be a probability density function on $\bbr$, and let $\theta=(\mu,\sigma)$ with $\mu\in\bbr$ and $\sigma>0$. The two-para statistical model
\begin{align*}
  p(x;\theta)=\frac{1}{\sigma}f\!\left(\frac{x-\mu}{\sigma}\right)
\end{align*}
is called the location-scale model, and we write $\mathcal{M}_f=\qty{p(x;\theta)\mid \theta\in\bbr\times\bbr_{>0}}$.
Here, $\mu$ is the location parameter and $\sigma$ is the scale parameter. 

For location-scale models, the Wasserstein distance \eqref{Wasserstein_distance} admits a closed form; moreover, they are totally geodesic in $\mathcal P_2(\bbr)$ with respect to the $\mathrm{L}^2$-Wasserstein distance (i.e., closed under displacement interpolation) \cite{Gelbrich}. 

In the location-scale family, the Wasserstein score functions are
\begin{align}
    \Phi_\mu^W(x ; \theta) &= x - \mathbb{E}_\theta[X], \label{score_mu} \\
    \Phi_\sigma^W(x ; \theta) &= \frac{(x - \mu)^2}{2\sigma} - \frac{V_\theta[X] + (\mathbb{E}_\theta[X] - \mu)^2}{2\sigma},\label{score_sigma} 
\end{align}
and the Wasserstein information matrix $I^W(\theta)$ is
\begin{align}
  I^W(\theta)=\begin{pmatrix}
    1 & {\mathbb{E}_\theta[X]-\mu\over\sigma}\\
    {\mathbb{E}_\theta[X]-\mu\over\sigma} & {V_\theta[X]\over\sigma^2}
  \end{pmatrix}.\label{w_mat_lc}
\end{align}
Note that if $f$ has mean $0$ and variance $1$, then \eqref{w_mat_lc} is the identity matrix. For a general $f$, the reparametrization $\omega=(\mu+\sigma m,\ \sigma s)$ brings the model to mean $0$ and variance $1$ at $\omega=(0,1)$, where
\begin{align}
  m&:=\bbe_{(0,1)}[X],\\
  s^2&:=V_{(0,1)}[X].
\end{align}
Since \eqref{w_mat_lc} can be made the identity matrix, the Riemannian curvature associated with the Wasserstein metric $g_W$ vanishes over the model $\mathcal{M}_f$; it means $\mathcal{M}_f$ is intrinsically flat as a Riemannian manifold $(\mathcal{M}_f,g_W)$.

\section{Location-scale-shape models}\label{Location-scale-shape models}

As an extension of the location-scale model, we introduce a location-scale-shape model that has one-dimensional location. scale, and shape parameter. 

\begin{defi}(location-scale-shape model)
  Let $\Xi\subset\mathbb{R}$ be an open set containing $0$, and let $f\in C^4(\bbr)$ be a probability density function on $\bbr$.
The three-parameter statistical model 
\[
\mathcal{S}_f=\{p(x;\theta) \mid \theta\in\mathbb{R}\times\mathbb{R}_{>0}\times\Xi\}
\]
is called the \emph{location-scale-shape model} generated by $f$ if the probability density function $p(x;\theta)$ is written as
    \begin{align*}
      p(x;\theta) = \begin{cases} 
\frac{1}{\sigma}\left(1 + \xi \frac{x - \mu}{\sigma}\right)^{-1}f\left(\frac{1}{\xi}\log\left( 1 + \xi \frac{x - \mu}{\sigma} \right)\right), & \text{if } \xi \neq 0, \\
\frac{1}{\sigma}f\left(\frac{x-\mu}{\sigma}\right), & \text{if } \xi = 0.
\end{cases}
\end{align*}
As in the location-scale model, we call $\mu$ the location parameter and $\sigma$ the scale parameter. Also, we call $\xi$ the shape parameter.
\end{defi}

Since
\begin{align*}
\text{supp}\ p_{\theta}=\Bigl\{x\mid 1 + \xi \frac{x - \mu}{\sigma}\geq0,\ \frac{1}{\xi}\log\Bigl( 1 + \xi \frac{x - \mu}{\sigma} \Bigr)\in \text{supp}\ f \Bigr\},
\end{align*}
the support of the distribution strongly depends on the parameters. Note that we require $f$ to be of class $C^4(\bbr)$ for theoretical reasons. This assumption is minimal for the present framework and one has $f\in C^\infty(\bbr)$ in many important examples.

  The shape parameter specifies the shape of the probability density function $f(x)$. 
  As that shape varies, the associated location and scale also vary accordingly. Note that this framework does not contain all types of  "shape parameters" in statistics. 
  In Examples 1 and 2 below, we show figures in which only the shape parameter is varied.
  Moreover, for each fixed $\xi\in\Xi$, define
  \begin{align*}
    f_\xi(x)=\begin{cases} 
(1 + \xi x)^{-1}f\qty({1\over\xi}\log\qty( 1 + \xi x)), & \text{if } \xi \neq 0, \\
f(x), & \text{if } \xi = 0.
\end{cases}
  \end{align*}
  This determines a location-scale model $\mathcal{M}_{f_\xi}$. In other words, a location-scale-shape model is a statistical model that continuously deforms location-scale models. Geometrically, the family $\{\mathcal{M}_{f_\xi}\}_{\xi\in \Xi}$ endows $\mathcal{S}_f$ with a foliation $\mathcal{S}_f=\bigsqcup_\xi \mathcal{M}_{f_\xi}$.

Below, we present two examples of the location-scale-shape model. These models play a central role in extreme value theory. Other examples include the generalized normal and generalized logistic distributions introduced in \cite{Hosking}; these are used in hydrology (see, e.g., \cite{Atiem, Das, Hosking}).

\begin{ex}
  (Generalized Extreme Value Distribution)\\
  A probability distribution with density
    \[
p(x; \theta) = 
\begin{cases} 
\frac{1}{\sigma} \left(1+\xi \frac{x - \mu}{\sigma}\right)^{-{1\over\xi}-1} \exp \left( -\left( 1 + \xi  \frac{x - \mu}{\sigma} \right)^{-{1\over\xi}} \right), & \text{if } \xi \neq 0, \\
\frac{1}{\sigma} \exp \left( -\left( \frac{x - \mu}{\sigma} \right) \right) \exp \left( -\exp \left( -\left( \frac{x - \mu}{\sigma} \right) \right) \right), & \text{if } \xi = 0
\end{cases}
\]
 is called the generalized extreme value distribution $\mathrm{GEV}(\mu, \sigma, \xi)$.
  This distribution is a location–scale–shape model generated by the standard Gumbel distribution,
\begin{align*}
  f(z)=\exp(-\exp(-z)).
\end{align*}As will be shown later, the moment generating function $f$ plays an important role in the location-scale-shape model; the moment generating function of the standard Gumbel distribution is given by
  \begin{align*}
    M_f(t)=\Gamma(1-t), \quad t<1,
  \end{align*}
  where $\Gamma(x)$ denotes the Gamma function.
  The support of $\mathrm{GEV}(\mu, \sigma, \xi)$ is given by\begin{align*}
\begin{cases}
  x\in[\mu-{\sigma\over\xi},\infty) , & \text{if } \xi > 0\\
  x\in[-\infty,\infty) , & \text{if } \xi = 0\\
  x\in(-\infty,\mu-{\sigma\over\xi}] , & \text{if } \xi < 0.
\end{cases}
\end{align*}
Figure \ref{fig:Densities of GEV} shows how the density function changes as $\xi$ varies for $\mathrm{GEV}(0,1,\xi)$.
  \begin{figure}[h]
    \centering
    \includegraphics[keepaspectratio, scale=0.5]{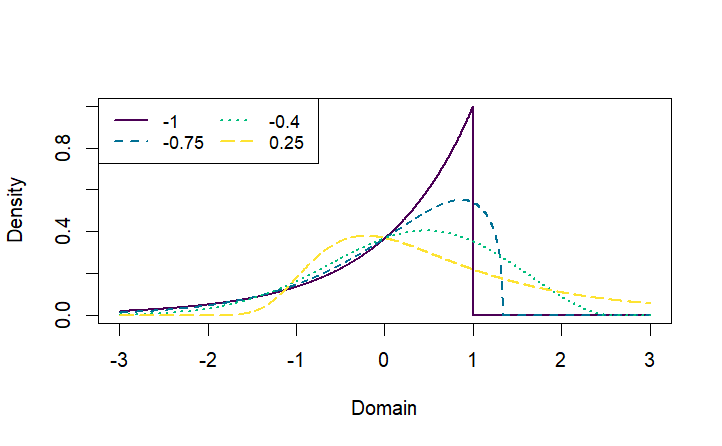}
    \caption{Densities of GEV.}
    \label{fig:Densities of GEV}
  \end{figure}
\end{ex}

\begin{ex}
  (Generalized Pareto Distribution)\\
  A probability distribution with density
  \[
p(x; \theta) = 
\begin{cases} 
\frac{1}{\sigma} \left( 1 + \xi \frac{x - \mu}{\sigma}  \right)^{-{1\over\xi} - 1}, & \text{if } \xi \neq 0, \\
\frac{1}{\sigma} \exp \left( -\left( \frac{x - \mu}{\sigma} \right) \right), & \text{if } \xi = 0.
\end{cases}
\]
is called the generalized Pareto distribution $\mathrm{GPD}(\mu, \sigma, \xi)$.
  This distribution is a location–scale–shape model generated by  the standard exponential distribution,
  \begin{align*}
    f(z)= \exp(-z).
  \end{align*}The moment generating function of the standard exponential distribution is given by
  \begin{align*}
    M_f(t)=\frac{1}{1-t}, \quad t<1.
  \end{align*}
  The support of $\mathrm{GPD}(\mu, \sigma, \xi)$ is given by\begin{align*}
\begin{cases}
  x\in[\mu,\infty) , & \text{if } \xi \geq 0,\\
  x\in[\mu,\mu-{\sigma\over\xi}] , & \text{if } \xi < 0.
\end{cases}
\end{align*}
  Figure \ref{fig:Densities of GPD} shows how the density function changes as $\xi$ varies for $\mathrm{GPD}(0,1,\xi)$. Note that $\mathrm{GPD}(\mu, \sigma, \xi)$ is uniform distribution when $\xi=-1$.
  \begin{figure}[h]
    \centering
    \includegraphics[keepaspectratio, scale=0.5]{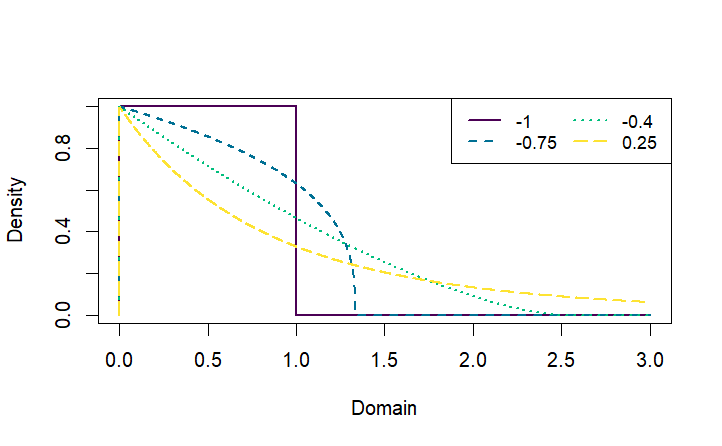}
    \caption{Densities of GPD.}
    \label{fig:Densities of GPD}
  \end{figure}
\end{ex}

The following propositions demonstrate that the shape parameter $\xi$ controls the tail heaviness of the distribution.
We denote by $F$ the cumulative distribution function of $f$, and by $P_\theta(x)$ the cumulative distribution function of $p(x;\theta)$.

\begin{prop}\label{prop_cum}
  $P_\theta(x)$ is nonincreasing in $\xi$ for fixed $x\in\bbr,\mu\in\bbr$, and $\sigma>0$ .
\end{prop}

\begin{proof}
Note that $P_\theta(x)$ is given by
\begin{align*}
  P_\theta(x) = \begin{cases} 
  F\qty(\log\qty( 1 + \xi  \frac{x - \mu}{\sigma} )^{1\over\xi}), & \text{if } \xi \neq 0, \\
  F\!\left(\frac{x-\mu}{\sigma}\right), & \text{if } \xi = 0,
\end{cases}
\end{align*}
for $x\in\text{supp}\, p_{\theta}$. Let $s={x-\mu\over\sigma}$.
If $\xi$ satisfies $1+\xi s>0$, we have
\begin{align*}
  \pdv{\xi}P_\theta(x)
  &= \pdv{\xi}F\qty(\log (1+\xi s)^{1\over\xi})\\
  &= \frac{1}{\xi^2}\,F'\qty(\log (1+\xi s)^{1\over\xi})\qty(\frac{\xi s}{1+\xi s}-\log (1+\xi s))<0,
\end{align*}
where $F'$ denotes the derivative of $F$, and the final inequality holds because
cumulative distribution functions are nondecreasing and $\frac{y}{1+y}-\log(1+y)<0$ for $y>-1$.

If $\xi$ satisfies $1+\xi s<0$, it follows that $x\notin \text{supp}\, p_{\theta}$, and hence $\pdv{\xi}P_\theta(x)=0$.

\end{proof}

\begin{prop}\label{prop_rate}
  Let  $\bar P_\theta(x)=1-P_\theta(x)$. Then, $\xi_1>\xi_2>0$ implies $\bar P_{\theta_2}(x)= O(\bar P_{\theta_1}(x)) \ (x\to\infty)$ where $\theta_i=(\mu_i, \sigma_i, \xi_i) $ for $ i=1,2$.
\end{prop}

\begin{rem}
  Similarly, $\xi_1<\xi_2<0$ implies $P_{\theta_2}(x)= O(P_{\theta_1}(x)) \ (x\to-\infty)$. Therefore, a positive $\xi$ corresponds to the right tail heavier, while a negative $\xi$ corresponds to the left tail heavier.
\end{rem}

\begin{proof}
 Let $\bar F(x):= 1- F(x)$ , which is a nonincreasing function. For $\xi_1>\xi_2>0$, we have
\begin{align*}
  \log\qty( 1 + \xi_1  \frac{x - \mu_1}{\sigma_1} )^{1\over\xi_1}-\log\qty( 1 + \xi_2  \frac{x - \mu_2}{\sigma_2} )^{1\over\xi_2}
  =\Bigl(\frac{1}{\xi_1}-\frac{1}{\xi_2}\Bigr)\log x+ O(1)\to-\infty \quad (x\to\infty).
\end{align*}
This implies that for sufficiently large $x$, we have
$\bar P_{\theta_2}(x)\leq \bar P_{\theta_1}(x)$. 
As both $\bar P_{\theta_1}(x)$ and $\bar P_{\theta_2}(x)$ tend to $0$, this inequality implies
$\bar P_{\theta_2}(x)= O\!\bigl(\bar P_{\theta_1}(x)\bigr)$ as $x\to\infty$.
\end{proof}

The following lemma simplifies the computation of moments.

\begin{lem}\label{lem_rv}
Let $\mathcal{S}_f=\{p(x;\theta)\mid\theta\in\mathbb{R}\times\mathbb{R}_{>0}\times\Xi\}$, $Z\sim f(z)$ and $X\sim p(x;\theta)$. Then for $\xi\neq0$, the following holds:
\begin{align}\label{law_of_Z}
X\overset{d}{=}&\frac{\sigma}{\xi}(\exp(\xi Z)-1)+\mu, 
\end{align}
Here, $\overset{d}{=}$ denotes that the random variables have the same distribution.
\end{lem}
\begin{proof}
Let $g(z):=\frac{\sigma}{\xi}(\exp(\xi z)-1)+\mu$ and set $W:=g(Z)$. We will show that $X\overset{d}{=}W$. Noting that the mapping $z\mapsto g(z)$ is one-to-one, we have
\begin{align*}
z=&\frac{1}{\xi}\log\left( 1+\xi\frac{w-\mu}{\sigma} \right),\\
\left|\frac{dz}{dw}\right|=&\frac{1}{\sigma}\left(1+\xi\frac{w-\mu}{\sigma}\right)^{-1}.
\end{align*}
Hence, the probability density function $q(w)$ of $W$ is given by
\begin{align*}
q(w)=\frac{1}{\sigma}\left(1+\xi\frac{w-\mu}{\sigma}\right)^{-1}f\left(\frac{1}{\xi}\log\left( 1+\xi\frac{w-\mu}{\sigma} \right)\right).
\end{align*}
This coincides with $p(x;\theta)$. Therefore $X\overset{d}{=}W$.
\end{proof}

From Lemma 1, the calculation of moments of $X$ reduces to the computation of the moment generating function of $Z$. For a set $A\subset\mathbb{R}$ and $r\in\mathbb{R}$, we denote\begin{align*}
  rA:=\qty{ra\mid a\in A}
  \end{align*}
  and for $f$, we denote the moment generating function 
  \begin{align*}
  M_f(t):=\int_\bbr e^{tz}f(z)\dd z.
\end{align*}
Then, from Lemma \ref{lem_rv}, we obtain
\begin{align}
  \mathbb{E}_\theta[X] &= \frac{\sigma}{\xi}(M_f(\xi)-1)+\mu,\\
  V_\theta[X] &= \frac{\sigma^2}{\xi^2}(M_f(2\xi)-M_f(\xi)^2).
\end{align}
For computational convenience, we often use the random variable $T=1+\xi\frac{X-\mu}{\sigma}$.  
Again from Lemma \ref{lem_rv}, we have
\begin{align}\label{law_of_T}
T 
&\overset{d}{=} 1+\frac{\xi}{\sigma}\left(\frac{\sigma}{\xi}(\exp(\xi Z)-1)+\mu-\mu\right)
= \exp(\xi Z).
\end{align}
The moments of $T$ satisfy the following relation.

\begin{cor}\label{cor_mom}
Let $k\in\mathbb{N}$ and $r\ge 0$, and assume $\xi\neq 0$. If the moment generating function $M_f(t)$ is defined and $k$ times differentiable on $r\Xi$, then
\begin{align}\label{mom_of_T}
\mathbb{E}_\theta[T^r(\log T)^k]=\xi^k M_f^{(k)}(r\xi).
\end{align}
Here, $M_f^{(k)}(t)$ denotes the $k$th derivative of $M_f(t)$.
\end{cor}
\begin{proof}
As a consequence of (\ref{law_of_T}) , we obtain
\begin{align*}
\mathbb{E}_\theta[T^r (\log T)^k]=&\mathbb{E}[\exp(r\xi Z)(\xi Z)^k]\\
=&\xi^k\mathbb{E}[Z^k\exp(r\xi Z)]\\
=&\xi^k M_f^{(k)}(r\xi).
\end{align*}
\end{proof}

We derive the Wasserstein score function in location-scale-shape models.

\begin{thm}\label{thm_wscore_functions}
In a location-scale-shape model $\mathcal{S}_f$, suppose 
$M_f(t)$ exists on $2\Xi$. Then the Wasserstein score function for the parameter $\mu$ is given by
\begin{align}\label{wscore_of_mu}
  \Phi_\mu^W(x ; \theta) &= x - \mathbb{E}_\theta[X], 
\end{align}
the Wasserstein score function for $\sigma$ is 
\begin{align}
  \Phi_\sigma^W(x ; \theta) &= \frac{(x - \mu)^2}{2\sigma} - \frac{V_\theta[X] + (\mathbb{E}_\theta[X] - \mu)^2}{2\sigma}, \label{wscore_of_sigma}
\end{align}
and the Wasserstein score function for  $\xi$ is
\begin{align}
  \Phi_\xi^W(x ; \theta) &= \begin{cases} 
\frac{\sigma^2}{\xi^3} \left( \frac{1}{2}t^2 \log t - \frac{3}{4}t^2 + t - \frac{\xi}{2} M_f'(2\xi)+\frac{3}{4}M_f(2\xi)-M_f(\xi)\right), & \text{if } \xi \neq 0, \\
 \frac{\sigma^2}{6}\left(\left(\frac{x-\mu}{\sigma}\right)^3-M_f'''(0)\right), & \text{if } \xi = 0,
  \end{cases}\label{wscore_of_xi}
\end{align}
where $t = 1 + \xi \frac{x - \mu}{\sigma}$ and $M_f'(w)$ and $M_f'''(w)$ denote the first and third derivatives of the moment generating function, respectively.

\end{thm}

\begin{rem}
From this result, it follows that the Wasserstein score function $\Phi^W_\xi$ for the shape parameter $\xi$ reflects information about the moments of order three or higher of $f$, and in particular, in the neighborhood of $\xi=0$ it conveys information about the third moment.
\end{rem}

\begin{rem}
The Fisher  score function $\Phi^F_i(x;\theta)=\frac{\partial \log p(x; \theta)}{\partial \theta^i}$ depends strongly on $f$, whereas the Wasserstein score function, apart from a constant term, does not depend on $f$.
\end{rem}

\begin{proof}
Since $\mathcal{S}_f$ becomes a location-scale model when $\xi$ is fixed, and noting that $\xi$ is fixed in the partial differential equations for $\Phi^W_\mu$ and $\Phi^W_\sigma$, equations \eqref{wscore_of_mu} and \eqref{wscore_of_sigma} follow from the corresponding results \eqref{score_mu} and \eqref{score_sigma} for the Wasserstein score functions in the location-scale model.

Next, for $\xi\neq0$ we verify by direct computation that $\Phi^W_\xi$ is a solution to the differential equation
\begin{align}\label{PDE}
  \frac{\partial}{\partial x}\qty(p(x;\theta){\del\Phi^W_\xi\over\del x})=-\frac{\partial}{\partial \xi}p(x;\theta).
\end{align}
First, the derivative of $\Phi^W_\xi$ with respect to $x$ is
\begin{align*}
  \frac{\partial}{\partial x}\Phi^W_\xi(x;\theta)
  =& \frac{\partial t}{\partial x}\frac{\partial}{\partial t}\Phi^W_\xi(t;\theta)\\
  =& \frac{\sigma}{\xi^2}(t\log t-t+1).
\end{align*}

Note that for $\xi\neq0$, $p(x;\theta)$ can be written as a function of $t$ by
\begin{align*}
  p(t;\theta)=\frac{1}{\sigma t}f\left(\log t^{1/\xi}\right).
\end{align*}
Denoting by $f'(z)$ the derivative of $f(z)$, the left-hand side of \eqref{PDE} becomes
\begin{align*}
  \frac{\partial}{\partial x}\qty(p(x;\theta){\del\Phi^W_\xi\over\del x})
  =& \frac{\partial t}{\partial x}\frac{\partial}{\partial t}\left(\frac{1}{\xi^2}\Bigl(\log t-1+\frac{1}{t}\Bigr)f\left(\log t^{1/\xi}\right)\right)\\
  =& \frac{1}{\sigma\xi t}\left(\left(1-\frac{1}{t}\right)f\left(\log t^{1/\xi}\right)+\frac{1}{\xi}\left(\log t-1+\frac{1}{t}\right)f'\left(\log t^{1/\xi}\right)\right).
\end{align*}
On the other hand,
\begin{align*}
  \frac{\partial}{\partial \xi}p(x;\theta)
  =& \frac{\partial t}{\partial \xi}\frac{\partial}{\partial t}p(t;\theta)+\frac{\partial}{\partial \xi}p(t;\theta)\\
  =& \frac{t-1}{\xi}\cdot \frac{1}{\sigma}\left(-\frac{1}{t^2}f\left(\log t^{1/\xi}\right)+\frac{1}{\xi t^2}f'\left(\log t^{1/\xi}\right)\right)
  -\frac{1}{\sigma\xi^2 t}\log t\,f\left(\log t^{1/\xi}\right)\\
  =& -\frac{1}{\sigma\xi t}\left(\left(1-\frac{1}{t}\right)f\left(\log t^{1/\xi}\right)+\frac{1}{\xi}\left(\log t-1+\frac{1}{t}\right)f'\left(\log t^{1/\xi}\right)\right).
\end{align*}
Therefore, we obtain
\begin{align*}
  \frac{\partial}{\partial x}p(x;\theta){\del\Phi^W_\xi\over\del x}
  = -\frac{\partial}{\partial \xi}p(x;\theta).
\end{align*}

Next, for $\xi\neq0$ we show that $\mathbb{E}[\Phi^W_\xi]=0$. Let $Z\sim f(z)$, $X\sim p(x;\theta)$, and $T=1+\xi\frac{X-\mu}{\sigma}$, then by corollary \ref{cor_mom} we have
\begin{align*}
  \mathbb{E}\Bigl[\frac{T^2}{2} \log T - \frac{3 T^2}{4} + T\Bigr]
  = \frac{\xi}{2} M_f'(2\xi)-\frac{3}{4}M_f(2\xi)+M_f(\xi).
\end{align*}

For $\xi=0$, it suffices to show that the limit as $\xi\to0$ coincides with \eqref{wscore_of_xi}. To avoid complications, let $a=\frac{x-\mu}{\sigma}$. Then, by the Taylor expansion of the logarithm,
\begin{align*}
  \log (1+a\xi)=a\xi-\frac{a^2}{2}\xi^2+\frac{a^3}{3}\xi^3+O(\xi^4), \quad (\xi\to0),
\end{align*}
so that
\begin{align*}
\frac{1}{2} t^2\log t
  = \frac{1}{2}(1+a\xi)^2\log (1+a\xi)
  = \frac{a}{2}\xi+\frac{3}{4}a^2\xi^2+\frac{a^3}{6}\xi^3+O(\xi^4).
\end{align*}
Therefore,
\begin{align*}
  \frac{1}{2}t^2 \log t - \frac{3}{4}t^2 + t
  =& \frac{a}{2}\xi+\frac{3}{4}a^2\xi^2+\frac{a^3}{6}\xi^3
  -\frac{3}{4}(1+a\xi)^2+(1+a\xi)+O(\xi^4)\\
  =& \frac{1}{4}+\frac{1}{6}a^3\xi^3+O(\xi^4).
\end{align*}
On the other hand, since the moment generating function is four-times differentiable on $2\Xi$, by Taylor's theorem we have
\begin{align*}
  \frac{\xi}{2} M_f'(2\xi)
  =& \frac{1}{2}M_f'(0)\xi+M_f''(0)\xi^2+M_f'''(0)\xi^3+O(\xi^4),\\
  M_f(\xi)
  =& M_f(0)+M_f'(0)\xi+\frac{1}{2}M_f''(0)\xi^2+\frac{1}{6}M_f'''(0)\xi^3+O(\xi^4),\\
  \frac{3}{4}M_f(2\xi)
  =& \frac{3}{4}M_f(0)+\frac{3}{2}M_f'(0)\xi+\frac{3}{2}M_f''(0)\xi^2+M_f'''(0)\xi^3+O(\xi^4).
\end{align*}
Combining these, we obtain
\begin{align*}
    \frac{\xi}{2} M_f'(2\xi)-\frac{3}{4}M_f(2\xi)+M_f(\xi)
    =& \frac{1}{4}M_f(0)+\frac{1}{6}M_f'''(0)\xi^3+O(\xi^4)\\
    =& \frac{1}{4}+\frac{1}{6}M_f'''(0)\xi^3+O(\xi^4),
\end{align*}
where we used the fact that $M_f(0)=1$ by definition of the moment generating function.

Therefore,
\begin{align*}
  \Phi^W_\xi(x;\theta)
  =& \frac{\sigma^2}{\xi^3} \Bigl( \frac{1}{2}t^2 \log t - \frac{3}{4}t^2 + t - \frac{\xi}{2} M_f'(2\xi)+\frac{3}{4}M_f(2\xi)-M_f(\xi)\Bigr)\\
  =& \frac{\sigma^2}{\xi^3} \Bigl( \Bigl(\frac{1}{4}+\frac{1}{6}a^3\xi^3\Bigr)
  - \Bigl(\frac{1}{4}+\frac{1}{6}M_f'''(0)\xi^3\Bigr)
  +O(\xi^4)\Bigr)\\
  =& \frac{\sigma^2}{6}\Bigl(\Bigl(\frac{x-\mu}{\sigma}\Bigr)^3-M_f'''(0)\Bigr)+O(\xi).
\end{align*}
Hence, when $\xi=0$,
\begin{align*}
   \Phi^W_\xi(x;\theta)
   = \frac{\sigma^2}{6}\left(\Bigl(\frac{x-\mu}{\sigma}\Bigr)^3-M_f'''(0)\right).
\end{align*}
\end{proof}

\begin{thm}\label{thm_w_info_mat}
The Wasserstein information matrix of the location–scale–shape model $\mathcal{S}_f$ with respect to parametrization $\theta=(\mu,\sigma,\xi)$ is given as 

\begin{align}\label{w_info_mat_theta}
   \scriptstyle I^W(\theta)
   =&\begin{pmatrix}
   \scriptstyle1 & \scriptstyle{1\over\xi}(M_f(\xi)-1) & \scriptstyle\frac{\sigma}{\xi^2}(\xi M_f'(\xi)-M_f(\xi)+1) \\
   * & \scriptstyle{1\over\xi^2}(M_f(2\xi)-2M_f(\xi)+1) & \scriptstyle{\sigma\over\xi^3}\qty(
    \xi M_f'(2\xi)-\xi M_f'(\xi)-M_f(2\xi)+2M_f(\xi)-1
    )\\
   * & * &\scriptstyle {\sigma^2\over\xi^4}(
    \xi^2 M_f''(2\xi)-2\xi M_f'(2\xi)+2\xi M_f'(\xi)+M_f(2\xi)-2M_f(\xi)+1
    )
\end{pmatrix},
  \end{align}
  where we omit the lower-triangular entries since $I^W(\theta)$ is symmetric.
\end{thm}

\begin{proof}
  It follows from direct computation using 
  corollary \ref{cor_mom} and
  theorem \ref{thm_wscore_functions} (see Appendix for details).
\end{proof}

\begin{ex}
In the generalized extreme value distribution, by Theorem \ref{thm_w_info_mat} and $M_f(t)=\Gamma(1-t)$, the Wasserstein information matrix is
\begin{align}
   \hspace{-0.5cm} \scriptstyle I^W(\theta)
   =\begin{pmatrix}
   \scriptstyle1 & \scriptstyle{1\over\xi}(\Gamma(1-\xi)-1) & \scriptstyle\frac{\sigma}{\xi^2}(1-\xi\Gamma'(1-\xi)-\Gamma(1-\xi)) \\
   * & \scriptstyle{1\over\xi^2}(\Gamma(1-2\xi)-2\Gamma(1-\xi)+1) & \scriptstyle{\sigma\over\xi^3}\qty(2\Gamma(1-\xi)-\Gamma(1-2\xi)+\xi\Gamma'(1-\xi)-\xi\Gamma'(1-2\xi)-1
    )\\
   * & * &\scriptstyle {\sigma^2\over\xi^4}(
    \xi^2 \Gamma''(1-2\xi)+2\xi \Gamma'(1-2\xi)-2\xi \Gamma'(1-\xi)+\Gamma(1-2\xi)-2\Gamma(1-\xi)+1
    )
\end{pmatrix}.
  \end{align}
\end{ex}

\begin{ex}
In the generalized Pareto distribution, by Theorem \ref{thm_w_info_mat} and $M_f(t)=\frac{1}{1-t}$, the Wasserstein information matrix is
\begin{align}
   I^W(\theta)
   =\begin{pmatrix}
   1 & \frac{1}{1-\xi} & \frac{\sigma}{(1-\xi)^2} \\
   * & \frac{2}{(1-\xi)(1-2\xi)} & \frac{\sigma(3-4\xi)}{(1-\xi)^2(1-2\xi)^2}\\
   * & * & \frac{\sigma^2(6-8\xi)}{(1-\xi)^2(1-2\xi)^3}
\end{pmatrix}.
\end{align}

\end{ex}

\section{Intrinsic and extrinsic Wasserstein geometry of location-scale-shape models}\label{Riemannian geometry of location-scale-shape models}

In this section, we study the Wasserstein geometry of location–scale–shape models.
First, we introduce new coordinates $\omega=(\alpha,\beta,\xi)$:
\begin{align}\label{def_omega}
\omega=(\mu+\sigma m_\xi,\ \sigma s_\xi,\ \xi),
\end{align}
where
\begin{align}
m_\xi:=&\bbe_{(0,1,\xi)}[X]={1\over\xi}\bigl(M_f(\xi)-1\bigr), \label{def_m}\\
s_\xi^2:=&V_{(0,1,\xi)}[X]={1\over\xi^2}\bigl(M_f(2\xi)-M_f(\xi)^2\bigr). \label{def_s}
\end{align}
The map $\theta\mapsto\omega$ is a diffeomorphism from $\mathbb{R}\times\mathbb{R}_{>0}\times\Xi$ onto itself, and its inverse transformation is given by
\begin{align}\label{inv_map}
\omega=(\alpha,\beta,\xi)\mapsto\qty(\alpha-{m_\xi\over s_\xi}\,\beta,\ {1\over s_\xi}\beta,\ \xi).
\end{align}
In the coordinates $\omega$ , for each fixed $\xi$, the associated location-scale model has mean $\alpha$ and variance $\beta^2$.

\begin{prop}\label{prop_w_info_mat_omega}
  The Wasserstein information matrix of the location–scale–shape model $\mathcal{S}_f$ with respect to parametrization $\omega=(\alpha,\beta,\xi)$ is given as 
  \begin{align}\label{w_info_mat_omega}
     I^W(\omega)=&\begin{pmatrix}
   1 & 0 & 0 \\
   0 & 1 & 0
   \\
   0 &0& \beta^2 \psi(\xi)
\end{pmatrix}
  \end{align}
  where $\psi(\xi)$ is a function only depends on $\xi$.
\end{prop}
\begin{proof}
First, we show that the $I^W(\theta)$ given by \eqref{w_info_mat_theta} can be rewritten using $m_\xi$ and $s_\xi^2$ as follows:
\begin{align}\label{info_mat_rewrighting}
  I^W(\theta)=\begin{pmatrix}
   1 & m_\xi & \sigma m'_\xi \\
   * & s_\xi^2+m_\xi^2 & \sigma(s_\xi s'_\xi+m_\xi m'_\xi)\\
   * & * & I_{33}
\end{pmatrix},
\end{align}
where $m'_\xi:=\frac{d}{d\xi}m_\xi$, $s_\xi:=\sqrt{s_\xi^2}$, and $s'_\xi:=\frac{d}{d\xi}s_\xi$.

Secondly,  differentiate \eqref{inv_map}, we obtain the Jacobi matrix \begin{align}\label{Jacobian}
  {\del \theta\over\del\omega}=\begin{pmatrix}
    1 & -{m_\xi\over s_\xi} & {m_\xi s'_\xi-m'_\xi s_\xi\over s_\xi^2}\beta\\
    0 & {1\over s_\xi} & -{s'_\xi\over s^2_\xi}\beta\\
    0 & 0 & 1
  \end{pmatrix}.
\end{align}
Hence combining \eqref{info_mat_rewrighting} and \eqref{Jacobian}, we have\begin{align*}
   &I^W(\omega)\\
  =&{\del \theta\over\del\omega}^\top I^W(\theta(\omega)) {\del \theta\over\del\omega}\\
  =&\begin{pmatrix}
    1 & 0 & 0\\
    {-m_\xi\over s_\xi} & {1\over s_\xi} & 0\\
    {(m_\xi s'_\xi- m'_\xi s_\xi)\beta\over s_\xi^2} & {-s'_\xi\beta\over s^2_\xi} & 1
  \end{pmatrix}
  \begin{pmatrix}
   1 & m_\xi & \sigma m'_\xi \\
   * & s_\xi^2+m_\xi^2 & \sigma(s_\xi s'_\xi+m_\xi m'_\xi)\\
   * & * &I_{33}
\end{pmatrix}
\begin{pmatrix}
    1 & {-m_\xi\over s_\xi} & {(m_\xi s'_\xi- m'_\xi s_\xi)\beta\over s_\xi^2}\\
    0 & {1\over s_\xi} & {-s'_\xi\beta\over s^2_\xi}\\
    0 & 0 & 1
  \end{pmatrix}\\
=&\begin{pmatrix}
   1 & 0 & 0 \\
   0 & 1 & 0
   \\
   0 &0& I_{33}
   -{{s'_\xi}^2+ {m'_\xi}^2\over s_\xi^2}\beta^2
\end{pmatrix},
\end{align*}
where the last equality holds because $\sigma={1\over s_\xi}\beta$. Recall that\begin{align*}
  I_{33}=&{\sigma^2\over\xi^4}(
    \xi^2 M_f''(2\xi)-2\xi M_f'(2\xi)+2\xi M_f'(\xi)+M_f(2\xi)-2M_f(\xi)+1
    )\\
    =&{\beta^2\over\xi^4s_\xi^2}(
    \xi^2 M_f''(2\xi)-2\xi M_f'(2\xi)+2\xi M_f'(\xi)+M_f(2\xi)-2M_f(\xi)+1
    ),
\end{align*}
we conclude\begin{align*}
  I^W(\omega)=&\begin{pmatrix}
   1 & 0 & 0 \\
   0 & 1 & 0
   \\
   0 &0& \beta^2 \psi(\xi)
\end{pmatrix}
\end{align*}
where\begin{align}\label{psi}
  \psi(\xi)={1\over s_\xi^2}\qty({\xi^2 M_f''(2\xi)-2\xi M_f'(2\xi)+2\xi M_f'(\xi)+M_f(2\xi)-2M_f(\xi)+1\over\xi^4}-({s'_\xi}^2+ {m'_\xi}^2)).
\end{align} 
\end{proof}

To restate Proposition \ref{prop_w_info_mat_omega} in geometric terms, we introduce the following definition.
Let $(M,g)$ and $(N,h)$ be Riemannian manifolds, and let $\varphi\in C^\infty(M)$. Define a Riemannian metric $G$ on $M\times N$ by
\begin{align*}
    G=\pi_M^* g+\bigl(\varphi\circ \pi_M\bigr)^2\,\pi_N^* h,
\end{align*}
where $\pi_M$ and $\pi_N$ denote the natural projections from $M\times N$ to $M$ and $N$, respectively, and ${}^*$ denotes the pullback. Then $(M\times N, G)$ is called the warped product of $(M,g)$ and $(N,h)$ by $\varphi$. We often omit $G$ and write $M\times_\varphi N$.
Warped products often appear in differential geometry related to general relativity; for details, see, for example, \cite{O'neil}.

With this terminology, Proposition \ref{prop_w_info_mat_omega} can be restated as follows.

\begin{cor}\label{cor_warped_prod}
Define a function $\varphi:\bbr\times\bbr_{>0}\to\bbr$ by $\varphi(\alpha,\beta)=\beta$. Then the Riemannian manifold $(\mathcal{S}_f,g_W)$ is isometric to the warped product
$(\bbr\times\bbr_{>0})\times_\varphi \Xi$ of $(\bbr\times\bbr_{>0},g_{\bbr^2})$ and $(\Xi,g_\bbr)$ with respect to a function $\varphi$.
\end{cor}
Finally, we establish the flatness of location-scale-shape model.
\begin{thm}\label{thm_lsc_is_flat}
The Riemannian curvature of the Riemannian manifold $(\mathcal{S}_f,g_W)$ vanishes.
\end{thm}

\begin{rem}
    In the proof below, we construct an affine coordinate system in which the metric becomes the identity matrix. If one is familiar with the computation of Riemannian curvature for warped products, one can readily verify directly that the curvature vanishes.
\end{rem}

\begin{proof}
To prove flatness, it is enough to have, locally, a coordinate system in which the metric is the identity matrix.
Let $\psi(\xi)$ be as defined in \eqref{psi}.
Choose an open interval $I$ with $I\subset\Xi$ and $\int_I \sqrt{\psi(\xi)}\,\dd\xi<2\pi$. 
Set $U=\bbr\times\bbr_{>0}\times I$. Define a coordinate transformation from $(\alpha,\beta,\xi)\in U$ to new local coordinates $(u,v,w)$
by
\begin{align*}
  (u,v,w):=\bigl(\alpha,\ \beta\cos\theta(\xi),\ \beta\sin\theta(\xi)\bigr),
\end{align*}
where
\begin{align*}
  \theta(\xi):=\int_{\inf I}^{\xi}\sqrt{\psi(x)}\,\dd x.
\end{align*}

Then we have
\begin{align*}
  \frac{\partial(u,v,w)}{\partial(\alpha,\beta,\xi)}
=\begin{pmatrix}
1 & 0 & 0\\
0 & \cos\theta(\xi) & -\,\beta\,\sqrt{\psi(\xi)}\sin\theta(\xi)\\
0 & \sin\theta(\xi) & \;\;\beta\,\sqrt{\psi(\xi)}\cos\theta(\xi)
\end{pmatrix},
\end{align*}
and it immediately yields
\begin{align*}
\frac{\partial(u,v,w)}{\partial(\alpha,\beta,\xi)}^\top
\frac{\partial(u,v,w)}{\partial(\alpha,\beta,\xi)}
=I^W(\omega).
\end{align*}
Consequently, we obtain \begin{align*}
  I^W(u,v,w)=\frac{\partial(\alpha,\beta,\xi)}{\partial(u,v,w)}^\top I^W(\omega)\frac{\partial(\alpha,\beta,\xi)}{\partial(u,v,w)}=\mathrm{diag}(1,1,1).
\end{align*}
Noting that $\Xi$ can be covered by such intervals $I$, flatness follows.

\end{proof}

Next, we discuss the extrinsic geometry of $\cals_f$. Here, extrinsic flatness means closedness under displacement interpolation in optimal transport theory, i.e., being totally geodesic with respect to the $\mathrm{L}^2$-Wasserstein distance.

Let $F$ be the cumulative distribution function of $f$.
For simplicity, assume that $F$ has an inverse $F^{-1}$; otherwise, interpret $F^{-1}$ as the quantile function.
The cumulative distribution function of $p_\theta\in\cals_f$ is given by
\begin{align*}
      &P_\theta(x) = \begin{cases} 
F\qty(\log\qty( 1 + \xi  \frac{x - \mu}{\sigma} )^{1\over\xi}), & \text{if } \xi \neq 0, \\
F({x-\mu\over\sigma}), & \text{if } \xi = 0,
\end{cases}
    \end{align*}
    Since $P^{-1}_\theta(x)={\sigma\over\xi}(\exp(\xi F^{-1}(x))-1)+\mu$, OT map from $p_{\theta_1}$ to $p_{\theta_2}$ is given by \begin{align*}
      T_{\theta_1\to\theta_2}(x)=&P_{\theta_2}^{-1}(P_{\theta_1}(x))\\
      =& {\sigma_2\over\xi_2}(\exp(\xi_2 F^{-1}\qty(F\qty(\log\qty( 1 + \xi_1  \frac{x - \mu_1}{\sigma_1} )^{1\over\xi_1})))-1)+\mu_2\\
      =& {\sigma_2\over\xi_2}\qty(( 1 + \xi_1  \frac{x - \mu_1}{\sigma_1} )^{\xi_2\over\xi_1}- (1-\xi_2  \frac{ \mu_2}{\sigma_2})
      ).
    \end{align*}
    Therefore, the $\mathrm{L}^2$-Wasserstein Geodesics from $p_{\theta_1}$ to $p_{\theta_2}$ is written as\begin{align}
      P_t(x)=&\qty((1-t)x+tT_{\theta_1\to\theta_2}(x))_\# P_{\theta_1}\quad (t\in [0,1])\nonumber \\
      =&P_{\theta_1}(((1-t)\cdot+tT_{\theta_1\to\theta_2}(\cdot))^{-1}(x))\label{geodesics}
    \end{align}
    where \begin{align*}
      (1-t)x+tT_{\theta_1\to\theta_2}(x)=&1+t{\sigma_2\over\xi_2}\qty(( 1 + \xi_1  \frac{x - \mu_1}{\sigma_1} )^{\xi_2\over\xi_1}- (1+\xi_2  \frac{ x- \mu_2}{\sigma_2})).
    \end{align*}
\begin{prop}\label{prop_geodesics}
  For $p_{\theta_1},p_{\theta_2}\in\cals_f$, $\qty{p_t}_{t\in(0,1)}$ defined by \eqref{geodesics} satisfy $p_t\in\cals_f$ if and only if $\xi_1=\xi_2$. 
\end{prop}
\begin{proof}
If $\xi_1=\xi_2$, it follows from the fact that the location–scale model is totally geodesic with respect to the $\mathrm L^2$-Wasserstein distance.

Assume $p_t=p_{\theta_t}\in\mathcal{S}_f$ where $\theta_t=(\mu_t,\sigma_t,\xi_t)$. In this case, from \eqref{geodesics} and the definition of $T_{\theta_1\to\theta_2}$ we have
\begin{align*}
  P_t^{-1}(x)=&(1-t)P^{-1}_{\theta_1}(x) +tT_{\theta_1\to\theta_2}(P^{-1}_{\theta_1}(x))\\
  =& (1-t)P^{-1}_{\theta_1}(x) +tP^{-1}_{\theta_2}(x).
\end{align*}
Therefore,
\begin{align*}
  &\frac{\sigma_t}{\xi_t}\bigl(\exp(\xi_t F^{-1}(x))-1\bigr)+\mu_t\\
  =&(1-t)\left(\frac{\sigma_1}{\xi_1}\bigl(\exp(\xi_1 F^{-1}(x))-1\bigr)+\mu_1\right)
  +t\left(\frac{\sigma_2}{\xi_2}\bigl(\exp(\xi_2 F^{-1}(x))-1\bigr)+\mu_2\right).
\end{align*}
Here, from the linear independence of $\{e^{\alpha x}\}_{\alpha\in\bbr}$ it follows that $\xi_1=\xi_2$.
\end{proof}

\begin{cor}\label{cor_extrinsic_geom}
  The location–scale–shape model $\mathcal S_f$ is not a totally geodesic submanifold of $\mathcal P_2(\mathbb{R})$ with respect to the $\mathrm L^2$-Wasserstein distance.
\end{cor}
Corollary \ref{cor_extrinsic_geom} means that, although the location–scale–shape model is intrinsically flat by Theorem \ref{thm_lsc_is_flat} , it has non-zero embedding curvature when embedded into the full space of probability measures $\mathcal P_2(\mathbb{R})$.
In the figure \ref{fig:pair-minipage} , we show the difference between the intrinsic geodesics based on the Wasserstein information matrix and the extrinsic geodesics based on displacement interpolation for the generalized extreme value (GEV) distribution.
On the intrinsic side, all displayed distributions are GEV distributions, whereas on the extrinsic side, except for the initial and final distributions, the intermediate ones are not GEV distributions.

\begin{figure}[tbp]
  \centering
  \begin{minipage}[t]{0.48\textwidth}
    \centering
    \includegraphics[width=\linewidth]{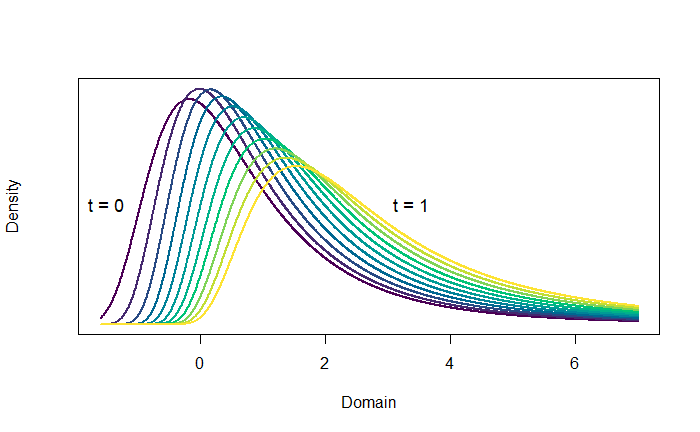}
  \end{minipage}
  \hfill
  \begin{minipage}[t]{0.48\textwidth}
    \centering
    \includegraphics[width=\linewidth]{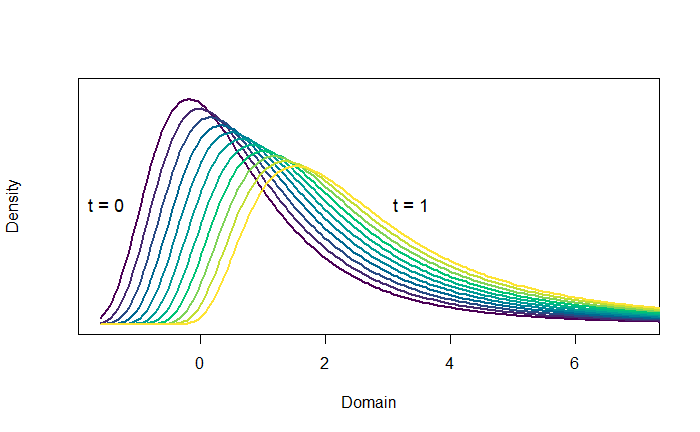}
  \end{minipage}
  \caption{Comparison of intrinsic (left) and extrinsic (right) geodesics from the purple distribution $\mathrm{GEV}(0,1,0.2)$, to the yellow distribution $\mathrm{GEV}(2,1.5,0.4)$.
}
  \label{fig:pair-minipage}
\end{figure}

\section{Discussion}\label{Discussions}

In this paper, we introduced location-scale-shape models and studied their Wasserstein geometry. We showed that, as in the location-scale models, the Wasserstein score functions do not depend on the base density $f$, and that the entries of the Wasserstein information matrix have simple expressions in terms of the moment-generating function of $f$. Although the parameterization combines location and scale parameters with a shape parameter, the resulting Wasserstein metric is not a direct product but rather a warped product. Moreover, the associated Riemannian manifold is flat. In contrast to the location-scale case, the model is not totally geodesic with respect to the $L^2$-Wasserstein distance.

Our analysis focused on geometric aspects and did not address the statistical properties of these models. For location-scale models, \cite{Amari and Matsuda2} studied the $Z$-estimator defined via the Wasserstein score function (Wasserstein estimator). In addition, \cite{Nishimori and Matsuda} showed that location-scale models form an $e$-geodesic in terms of the Wasserstein dual geometry proposed by \cite{Ay}, and discussed the relation between these geodesics and Wasserstein--Cramer--Rao  efficiency within the framework of \cite{Li and Zhao}. This result is analogous to the Cramer--Rao efficiency of the maximum likelihood estimator in exponential families. In location-scale models, the Wasserstein estimator admits a closed-form expression, which facilitates various derivations; by contrast, no such closed form is available for location-scale-shape models. It is an interesting future work to study the statistical properties of location-scale-shape models and their connections with Wasserstein geometry.

We restricted our attention to univariate models in this study; extension to the multivariate setting is left for future work. 
The Wasserstein geometry of the multivariate Gaussian model has been well elucidated \cite{Takatsu}.
This result can be extended to location-scatter models, a multivariate extension of the location-scale model, and Wasserstein statistics of the location-scatter model has been studied recently \cite{Amari and Matsuda2}. For location-scale-shape models, however, even an appropriate multivariate extension is nontrivial. We may be able to propose a natural multivariate extension from the perspective of Wasserstein geometry.

	\vskip 14pt
	\noindent {\large\bf Acknowledgements}
	Takeru Matsuda was supported by JSPS KAKENHI Grant Numbers 19K20220, 21H05205, 22K17865 and JST Moonshot Grant Number JPMJMS2024.
	
	\vskip 14pt
	
	\noindent {\large\bf Conflict of interest statement}
	On behalf of all authors, the corresponding author states that there is no conflict of interest. 
	
	\vskip 14pt
	
	\noindent {\large\bf Data availability statement}
	Not Applicable.
	
	\vskip 14pt

\begin{appendices}
\section{Proof of Theorem \ref{thm_w_info_mat}}
Below, for simplicity, we assume $\xi\neq0$. Note, however, that all functions that appear are continuous at $\xi=0$, so the limit $\xi\to\pm0$ can be taken. We then compute the Wasserstein information matrix for the location-scale-shape model $\mathcal{S}_f$.

By Theorem \ref{thm_wscore_functions}, the derivatives with respect to $x$ of the Wasserstein score functions in the location-scale-shape model are given by the following.

\begin{align*}
     {\del \over \del x} \Phi_\mu^W(x ; \theta) &= 1,\\
      {\del \over \del x}\Phi_\sigma^W(x ; \theta) &= \frac{x - \mu}{\sigma}\\
      &={t-1\over\xi},\\
      {\del \over \del x}\Phi_\xi^W(x ; \theta) &={\del t\over \del x}{\del \over \del t}\qty(\frac{\sigma^2}{\xi^3} ( \frac{t^2}{2} \log t - \frac{3 t^2}{4} + t)) \\
      &=\frac{\sigma}{\xi^2} \qty( t \log t -t + 1),
    \end{align*}
where $t=1+\xi{x-\mu\over\sigma}$.

Using \eqref{law_of_Z}, \eqref{law_of_T}
and Corollary \ref{cor_mom}, we can express the following in terms of the moment generating function $M_f(t)$:

\begin{align*}
      I^W(\theta)_{11}=&\mathbb{E}_\theta[1\cdot1]=1,\\
    I^W(\theta)_{12}=&\mathbb{E}_\theta\qty[1\cdot{1\over\xi}(T-1)]\\
    =& {1\over\xi}(M_f(\xi)-1), \\
    I^W(\theta)_{22}=&\mathbb{E}_\theta\qty[\qty({1\over\xi}(T-1))^2]\\
    =&{1\over\xi^2}(M_f(2\xi)-2M_f(\xi)+1),\\
    I^W(\theta)_{13}=&\mathbb{E}_\theta[\frac{\sigma}{\xi^2} \qty( T \log T -T + 1)]\\
    =&\frac{\sigma}{\xi^2}(\xi M_f'(\xi)-M_f(\xi)+1),\\
    I^W(\theta)_{23}=&\mathbb{E}_\theta\qty[({1\over\xi}(T-1)) (\frac{\sigma}{\xi^2} \qty( T \log T -T + 1))]\\
    =&{\sigma\over\xi^3}\qty(
    \xi M_f'(2\xi)-\xi M_f'(\xi)-M_f(2\xi)+2M_f(\xi)-1
    ),\\
      I^W(\theta)_{33}=&\mathbb{E}_\theta\qty[(\frac{\sigma}{\xi^2} \qty( T \log T -T + 1))^2]\\
    =&{\sigma^2\over\xi^4}(
    \xi^2 M_f''(2\xi)-2\xi M_f'(2\xi)+2\xi M_f'(\xi)+M_f(2\xi)-2M_f(\xi)+1
    ).
\end{align*}
Consequently, the Wasserstein information matrix of the location-scale-shape model is given by
\begin{align*}
   \scriptstyle I^W(\theta)
   =&\begin{pmatrix}
   \scriptstyle1 & \scriptstyle{1\over\xi}(M_f(\xi)-1) & \scriptstyle\frac{\sigma}{\xi^2}(\xi M_f'(\xi)-M_f(\xi)+1) \\
   * & \scriptstyle{1\over\xi^2}(M_f(2\xi)-2M_f(\xi)+1) & \scriptstyle{\sigma\over\xi^3}\qty(
    \xi M_f'(2\xi)-\xi M_f'(\xi)-M_f(2\xi)+2M_f(\xi)-1
    )\\
   * & * &\scriptstyle {\sigma^2\over\xi^4}(
    \xi^2 M_f''(2\xi)-2\xi M_f'(2\xi)+2\xi M_f'(\xi)+M_f(2\xi)-2M_f(\xi)+1
    )
\end{pmatrix}.
  \end{align*}\qed
\end{appendices}

\end{document}